\documentclass{article}
\usepackage{multicol,graphicx,color}
\usepackage{pslatex}
\usepackage{authblk}
\usepackage{amsthm}
\usepackage{amsmath}
\usepackage{amssymb}
\usepackage{latexsym}
\usepackage{lscape}
\usepackage{epsfig}
\usepackage{pstricks}
\usepackage{amsfonts}
\usepackage{mathrsfs}
\usepackage{mathrsfs}
\usepackage[
  hmarginratio={1:1},     
  vmarginratio={1:1},     
  textwidth=15cm,        
  textheight=21cm,
  heightrounded,          
]{geometry}

\usepackage{graphicx,color}
\usepackage[colorlinks]{hyperref}
\hypersetup{linkcolor=blue,citecolor=blue,filecolor=black,urlcolor=blue}

\theoremstyle{plain}
\newtheorem{theorem}{Theorem}

\newtheorem{lemma}[theorem]{Lemma}
\newtheorem{proposition}[theorem]{Proposition}
\theoremstyle{definition}
\newtheorem{definition}[theorem]{Definition}
\newtheorem{remark}[theorem]{Remark}

\numberwithin{equation}{section}
\numberwithin{theorem}{section}

\newcommand{\M}{S_\mu}

\newcommand{\cG}{\mathcal{G}}
\newcommand{\R}{\mathbb{R}}
\newcommand{\N}{\mathbb{N}}
\newcommand{\eps}{\varepsilon}
\newcommand{\dist}{{\rm dist}}

\newcommand{\edge}{{\rm e}}
\newcommand{\pa}{\partial}

\author{Xiaojun Chang \footnote{changxj100@nenu.edu.cn}} \affil{School of Mathematics and Statistics \& Center for Mathematics and Interdisciplinary Sciences,
 Northeast Normal University, Changchun 130024, Jilin,
PR China}

\author{Louis Jeanjean\footnote{louis.jeanjean@univ-fcomte.fr}}  \affil{Laboratoire de Math\'ematiques (CNRS UMR 6623), Universit\'e de Franche-Comt\'e,
Besancon 25030, France}

\author{Nicola Soave\footnote{nicola.soave@polimi.it}}  \affil{Dipartimento di Matematica, Politecnico di Milano, Via Bonardi 9, 20133, Milano, Italy}

\title{Normalized solutions of $L^2$-supercritical NLS equations on compact metric graphs}
\date{}

\begin{document}

\maketitle

\begin{abstract}
\noindent This paper is devoted to the existence of non-trivial bound states of prescribed mass for the mass-supercritical nonlinear Schr\"odinger equation on compact metric graphs. The investigation is based upon a variational principle which combines the monotonicity trick and a min-max theorem with second order information for constrained functionals, and upon the blow-up analysis of bound states with prescribed mass and bounded Morse index.
\end{abstract}

\medskip

{\small \noindent \text{Key Words:} Nonlinear Schr\"odinger equations; $L^2$-supercritical;
Compact metric graph; Variational methods.}

\medskip

{\small \noindent \text{Acknowledgements:}
X. J. Chang is partially supported by NSFC (11971095).
N. Soave is partially supported by the INDAM-GNAMPA group.}

\section{Introduction and main results}\label{intro}

In this paper we investigate the existence of non-constant critical points for the \emph{mass supercritical} NLS energy functional $E(\cdot\,,\mathcal{G}):H^1(\cG) \to \R$ defined by
\begin{eqnarray}\label{1.1}
E(u,\mathcal{G})=\frac{1}{2}\int_{\mathcal{G}}|u'|^2\, dx-\frac{1}{p}\int_{\mathcal{G}}|u|^p\, dx, \qquad p>6
\end{eqnarray}
under the mass constraint
\begin{equation}\label{mass const}
\int_{\mathcal{G}}|u|^2\,dx=\mu>0,
\end{equation}
where $\mathcal{G}$ is a \emph{compact metric graph}. Critical points, also called \emph{bound states}, solve the stationary nonlinear Schr\"odinger equation (NLS) on $\cG$
\begin{equation}\label{stat nls}
-u''+\lambda u=|u|^{p-2}u
\end{equation}
for some Lagrange multiplier $\lambda$, coupled with Kirchhoff condition at the vertexes (see \eqref{1.2} below). In turn, solutions to \eqref{1.2} give standing waves of the time-dependent NLS on $\cG$.

There are several physical motivations to consider Schr\"odinger equations on metric graphs. We refer the interested reader to the recent paper \cite{KNP}, to \cite{AST11, BK, No}, and to the references therein. In addition, the problem on metric graphs presents interesting new mathematical features with respect to the Euclidean case. For these reasons, the problem of existence of bound states on metric graphs attracted a lot of attention in the past decade, mainly in the \emph{subcritical} or \emph{critical regimes}, which correspond to $p \in (2,6)$ or $p=6$, respectively. In such frameworks, a particularly relevant issue concerns the existence of \emph{ground states}, that is, global minimizers of the energy under the mass constraint, see \cite{ACFN14, AST-CVPDE2015,AST-JFA2016,AST-CMP2017} for non-compact $\cG$, and \cite{CDS, D-JDE2018} for the compact case. We also refer the interested reader to \cite{BD19, DTcalcvar, NP, PiSo, PSV, ST-JDE2016} and references therein for strictly related issues (problems with localized nonlinearities, combined nonlinearities, existence of critical points in absence of ground states), always in subcritical and critical regimes.

In striking contrast, as far as we know, the supercritical regime on general graphs is essentially untouched. In this case the energy is always unbounded from below, and ground states never exist. However, it is natural to discuss the existence of bound states, and in this paper we address this problem on any \emph{compact} graph $\cG$. An interesting feature of this setting is that there always exists a constrained constant (trivial) critical point of $E(\cdot\,,\cG)$, obtained by taking the constant function $\kappa_\mu:= (\mu/\ell)^{1/2}$, where $\ell$ denotes the total length of $\cG$. Thus, in order to obtain a non-trivial result, one has to focus on existence of non-constant bound states.
\subsection*{Basic notations and main result}
We recall that a metric graph $\cG = (\mathcal{E}, \mathcal{V})$ is a connected metric space obtained by glueing together a number of closed line intervals, the edges in $\mathcal{E}$, by identifying some of their endpoints, the vertexes in $\mathcal{V}$. The peculiar way in which these identifications are performed defines the topology of $\cG$. Any bounded edge ${\rm e}$ is identified with a closed bounded interval $I_{\rm e}$, typically $[0,\ell_{\rm e}]$ (where $\ell_{\rm e}$ is the length of ${\rm e}$), while unbounded edges are identified with (a copy of) the closed half-line $[0,+\infty)$. A metric graph is compact if and only if it has a finite number of edges, and none of them is unbounded.

A function $u$ on $\cG$ is a map $u: \cG \to \R$, which is identified with a vector of functions $\{u_{\mathrm{e}}\}$, where each $u_{\mathrm{e}}$ is defined on the corresponding interval $I_{\rm e}$. Endowing each edge with Lebesgue measure, one can define $L^p$ spaces over $\cG$, denoted by $L^p(\cG)$, in a natural way, with norm
\[
\|u\|_{L^p(\cG)}^p = \sum_{\mathrm{e}} \|u_{\mathrm{e}}\|_{L^p(\mathrm{e})}^p.
\]
The Sobolev space $H^1(\cG)$ is defined as the set of functions $u: \cG \to \R$ such that $u_{\mathrm{e}} \in H^1([0, \ell_{\mathrm{e}}])$ for every bounded edge $\mathrm{e}$, $u_{\mathrm{e}} \in H^1([0, +\infty))$ for every unbounded edge $\mathrm{e}$, and $u$ is continuous on $\cG$ (in particular, if a vertex $\mathrm{v}$ belongs to two or more edges $\mathrm{e}_i$, the corresponding functions $u_{\mathrm{e}_i}$ take the same value on $\mathrm{v}$); the norm in $H^1(\cG)$ is naturally defined as
\[
\|u\|_{H^1(\cG)}^2 = \sum_{\mathrm{e}} \|u_{\mathrm{e}}'\|_{L^2(\mathrm{e})}^2 + \|u_{\mathrm{e}}\|_{L^2(\mathrm{e})}^2.
\]

We aim at proving the existence of non-constant critical points of the energy $E(\cdot\,, \cG)$, defined in \eqref{1.1}, constrained on the $L^2$-sphere
\[
H_{\mu}^1(\mathcal{G}):= \left\{u\in H^1(\mathcal{G}): \int_{\mathcal{G}}|u|^2\,dx=\mu\right\}.
\]
If $u\in H^1_{\mu}(\mathcal{G})$ is such a critical point, then there exists a Lagrange multiplier $\lambda\in \mathbb{R}$ such that $u$ satisfies the following problem:
\begin{equation}\label{1.2}
\begin{cases}
-u''+\lambda u=|u|^{p-2}u~~&\mbox{for every edge}~ {\rm e}\in \mathcal{E},\\
\sum\limits_{\edge \succ {\rm v}}u_{\edge}'({\rm v})=0~~&\mbox{at every vertex}~{\rm v} \in \mathcal{V},
\end{cases}
\end{equation}
where $\edge \succ {\rm v}$ means that the edge $\edge$ is incident at ${\rm v}$, and the derivative $u'_{\edge}(v)$ is always an outer derivative. The second equation is the so called \emph{Kirchhoff condition}. Notice that the positive constant function $\kappa_\mu= (\mu/\ell)^{1/2}$ trivially satisfies \eqref{1.2}, for $\lambda = (\mu/\ell)^{(p-2)/2}$.

Our main existence result is as follows.
\begin{theorem}\label{thm: main ex}
Let $\cG$ be any compact metric graph, and $p>6$. There exists $\mu_1>0$ depending on $\cG$ and on $p$ such that, for any $0<\mu<\mu_1$, problem \eqref{1.2} with the mass constraint \eqref{mass const} has a positive non-constant solution which corresponds to a mountain pass critical point of $E(\cdot\,,\mathcal{G})$ on $H_{\mu}^1(\mathcal{G})$, at a strictly larger energy level than $\kappa_\mu$.
\end{theorem}

\begin{remark}
Note that the Lagrange multiplier associated with any positive solution $u$ to \eqref{1.2} is positive. Indeed, by standard arguments, we know that $u \in C^2(\edge)$ on every edge. Then, integrating the first equation in \eqref{1.2} on every edge, summing over the edges and making use of the Kirchhoff condition, we obtain
\[
\lambda \|u\|_{L^1(\cG)} = \|u\|_{L^{p-1}(\cG)}^{p-1},
\]
whence we deduce that $\lambda >0$.
\end{remark}

\begin{remark}
The theorem is not a perturbation result, in the sense that the value $\mu_1$ will not be obtained by any limit process, and can be explicitly estimated. We refer to Proposition \ref{constant-solution} and Remark \ref{rem: on mu_1} for more details.

On the other hand, one may wonder whether or not the restriction $\mu<\mu_1$ can be removed. This is an open problem, our min-max approach fails for large masses. Observing that our solutions will have Morse index at most $2$ as critical points of the associated action functional (see Section \ref{sec: appr}), another related issue could be to investigate if it is possible to find solutions of \eqref{1.2}, possibly non-positive, with any mass $\mu>0$ and Morse index bounded by $2$. For the NLS equations with Dirichlet conditions in bounded Euclidean domains, this question has a negative answer, see \cite[Theorem 1.2]{PieVer}. Even if the two problems are not equivalent, this result suggests that a bound of type $\mu<\mu_1$ may be necessary.
\end{remark}

The proof of Theorem \ref{thm: main ex} is divided into some intermediate steps. At first, in Section \ref{sec: min}, we observe the local minimality of the constant solution for $\mu<\mu_1$, following \cite{CDS}.

Since in addition $E(\cdot\,,\cG)$ is unbounded from below, as $p>6$, this naturally suggests the possible existence of a second critical point, of mountain pass type. However, we have to face some severe compactness issues, and, in particular, the existence of a \emph{bounded} Palais-Smale sequence at the mountain pass level is not straightforward. We point out that the techniques based on scalings, usually employed in the Euclidean setting and related to the validity of a Pohozaev identity (see \cite{J} or \cite{BaSo, IkNo}), do not work, since $\cG$ is not scale invariant. To overcome this obstruction, a first natural attempt is to adapt the monotonicity trick of \cite{J-PRSE1999}: we first introduce a family of functionals $E_\rho(\cdot\,,\cG): H^1_\mu(\cG) \to \R$ defined by
\[
E_\rho(u,\cG) = \frac{1}{2}\int_{\mathcal{G}}|u'|^2\, dx-\frac{\rho}{p}\int_{\mathcal{G}}|u|^p\, dx, \quad \rho \in \left[\frac12,1\right].
\]
 Exploiting the monotonicity of $E_\rho(u, \cG)$ with respect to $\rho$, we can easily show that $E_\rho(\cdot\,, \cG)|_{H^1_\mu(\cG)}$ has a bounded Palais-Smale sequence of mountain pass type, for almost every $\rho \in [1/2,1]$. Since $\cG$ is compact, this ensures the existence of a critical point $u_\rho$ of $E_\rho(u, \cG)$, for almost every $\rho \in [1/2,1]$. Now the idea is to take the limit of $\{u_{\rho_n}\}$ along a sequence $\rho_n \to 1^-$, in order to obtain a critical point of the original functional. However, once again the boundedness of $\{u_{\rho_n}\}$ is an issue. In order to gain compactness, we use a general principle which combines the monotonicity trick, as presented in \cite{J-PRSE1999}, and the mountain pass theorem with second order information for constrained functionals, Theorem \ref{thm: monot trick second order} below (this is \cite[Theorem 1]{BoChJeSo}). A similar result was recently proved in \cite{BeRu,LoMaRu} in the unconstrained setting.

Mountain pass or min-max theorems with second order information have been introduced in \cite{FG1, FG-1994}. The second order information turns out to be extremely useful in proving the compactness of Palais-Smale sequences when the problem is not scale-invariant (and hence a Pohozaev identity is not available). 


With Theorem \ref{thm: monot trick second order}, we prove the existence of a sequence $\{u_{\rho_n}\}$ critical points for $E_\rho(\cdot\,, \cG)|_{H^1_\mu(\cG)}$ (which in particular are solutions of approximating problems) \emph{with uniformly bounded Morse index}. In Section \ref{sec: blow-up}, we perform a detailed blow-up analysis for this type of sequences, in the spirit of \cite{EspPet} (see also \cite{PieVer}). We think that this analysis is of independent interest and, for the sake of generality, we perform it on graphs which are not necessarily compact. In Theorem \ref{thm: blow-up 1}, we characterize the blow-up behavior of solutions close to local maximum points, both when they accumulate in the interior of one edge, or when they accumulate on a vertex; in the latter case, the limit problem is
an NLS equation posed on a \emph{star-graph}, which is a new phenomenon with respect to the Euclidean case. In Theorem \ref{thm: blow-up 2}, we establish a relation between the upper bound on the Morse index and the number of maximum points of the solutions, and describe the behavior far away from them.

Afterwards, Theorems \ref{thm: blow-up 1} and \ref{thm: blow-up 2} are used in Section \ref{sec: ex mp} to finally deduce, via a contradiction argument, that also the sequence $\{u_{\rho_n}\}$ is bounded, and converges to the non-constant mountain pass solution of Theorem \ref{thm: main ex}.

\section{Local minimality of the constant solution}\label{sec: min}

Let $\kappa_{\mu}:=(\mu/\ell)^{1/2}$ with $\ell:=|\mathcal{G}|$ being the total length of the graph $\mathcal{G}$. Clearly, the constant function $\kappa_{\mu}$ is always a solution to (\ref{1.2}) in $H_{\mu}^1(\mathcal{G})$ for some $\lambda\in\mathbb{R}$, and hence a constrained critical point of $E(\cdot\,,\cG)$ on $H^1_\mu(\cG)$. Furthermore, following \cite{CDS}, we can give a variational characterization of $\kappa_{\mu}$.

\begin{proposition}\label{constant-solution}
Assume that $\mathcal{G}$ is a compact graph and $p>6$. Then there exists $\mu_1>0$ depending on $\cG$ and on $p$ such that
\begin{description}
\item[(i)]if $0 <\mu<\mu_1$, then $\kappa_{\mu}$ is a strict local minimizer of $E(u,\mathcal{G})$ in $H^1_{\mu}(\mathcal{G})$;
\item[(ii)]if $\mu>\mu_1$, then $\kappa_{\mu}$ is not a local minimizer of $E(u,\mathcal{G})$ in $H^1_{\mu}(\mathcal{G})$.
\end{description}
\end{proposition}
\begin{proof}
To characterize the variational properties of $\kappa_{\mu}$, we shall evaluate the sign of the quadratic form $\varphi \in T_{\kappa_\mu} H^1_\mu(\cG) \mapsto d^2|_{H^1_\mu(\cG)} E(\kappa_{\mu},\mathcal{G})[\varphi, \varphi]$, where $d^2|_{H^1_\mu(\cG)} E(u, \cG)$ denotes the constrained Hessian of $E(\cdot\,\cG)$ on $H^1_\mu(\cG)$ and $T_{\kappa_{\mu}}H_{\mu}^1(\mathcal{G})$ is the tangent space of $H_{\mu}^1(\mathcal{G})$ at $\kappa_{\mu}$, defined as follows:
\begin{eqnarray*}
T_{\kappa_{\mu}}H_{\mu}^1(\mathcal{G}):=\left\{\phi \in H^1(\mathcal{G}): \int_{\mathcal{G}}\phi \, dx=0\right\}.
\end{eqnarray*}
From \cite[Proposition 4.1]{CDS}, which remains valid with the same proof for $p>6$ we obtain
\begin{eqnarray}\label{2-23-1}
d^2|_{H^1_\mu(\cG)} E (\kappa_{\mu},\mathcal{G})[\phi,\phi]=\int_{\mathcal{G}}|\phi'|^2\,dx-(p-2)\kappa_{\mu}^{p-2}\int_{\mathcal{G}}|\phi|^2\, dx, \qquad \forall \phi\in T_{\kappa_{\mu}}H_{\mu}^1(\mathcal{G}).
\end{eqnarray}

Denote now by $\lambda_2(\cG)$ the smallest positive eigenvalue of the Kirchhoff Laplacian on $\cG$ (that is $-(\cdot)''$ on $\cG$, coupled with the Kirchhoff condition at the vertexes), namely
%
%
$$\lambda_2(\mathcal{G}) = \inf_{\phi \in H^1(\mathcal{G}), \int_{\mathcal{G}} \phi \, dx = 0} \frac{\int_{\mathcal{G}} |\phi '|^2 dx}{\int_{\mathcal{G}} |\phi|^2 dx}.$$
Let us suppose that  $0 <\mu<\mu_1$, where
\begin{equation}\label{defmu}
\mu_1 := \ell \, \Big( \frac{\lambda_2(\mathcal{G})}{p-2}\Big)^{\frac{2}{p-2}},
\end{equation}
and let $\beta \in (0,1)$ be such that
\begin{equation}\label{choose alpha}
\beta \lambda_2(\cG) -(p-2) \left(\frac{\mu}{\ell}\right)^{\frac{p-2}{2}}>0.
\end{equation}
In view of (\ref{2-23-1}), it follows that
\begin{eqnarray*}
d^2|_{H^1_\mu(\cG)}  E(\kappa_{\mu},\mathcal{G})[\phi,\phi] \ge (1-\beta) \int_{\cG} |\phi'|^2 +  [\beta \lambda_2(\mathcal{G})-(p-2)\kappa_{\mu}^{p-2}]\int_{\mathcal{G}}|\phi|^2dx
\end{eqnarray*}
for every $\phi\in T_{\kappa_{\mu}}H_{\mu}^1(\mathcal{G})$, which implies that $d^2|_{H^1_\mu(\cG)}  E(\kappa_{\mu},\mathcal{G})$ is positive definite whenever
$0 <\mu<\mu_1$. Hence, for any such $\mu$, the constant $\kappa_\mu$ is a strict local minimizer of $E(\cdot\,,\cG)$ on $H^1_\mu(\cG)$.
\smallskip

If instead $\mu>\mu_1$, taking an eigenfunction $\phi_2$ corresponding to $\lambda_2(\mathcal{G})$, we obtain
\begin{eqnarray*}
d^2|_{H^1_\mu(\cG)} E(\kappa_\mu, \cG)[\phi_2, \phi_2]=[\lambda_2(\mathcal{G})-(p-2)\kappa_{\mu}^{p-2}]\int_{\mathcal{G}}|\phi_2|^2dx<0,
\end{eqnarray*}
which implies that $\kappa_{\mu}$ is not a local minimizer of $E(u,\mathcal{G})$ in $H_{\mu}^1(\mathcal{G})$.
\end{proof}

\begin{remark}\label{rem: on mu_1}
By \cite[Theorem 1]{F-2005}, we have $\lambda_2(\mathcal{G}) \geq \pi^2/\ell^2$.
Then by \eqref{defmu} it follows that
$$\mu_1 \geq \ell^{\frac{p-6}{p-2}} \Big(\frac{\pi^2}{p-2}\Big)^{\frac{2}{p-2}}.$$
In particular, $\mu_1\to +\infty$ as $\ell \to+\infty.$
\end{remark}

\section{Mountain pass solutions for approximating problems}\label{sec: appr}

When $\kappa_{\mu}$ is a local minimizer of the energy, and since the energy is unbounded from below on $H^1_\mu(\cG)$ in the supercritical regime, one may consider the question of finding a non-constant solution of mountain pass (MP) type. The existence of a MP solution will be the content of this and the next two sections. Before proceeding, it is convenient to recall a preliminary result and a definition.

\begin{lemma}[Proposition 3.1 in \cite{D-JDE2018}]\label{Embedding}
Assume that $\mathcal{G}$ is a compact graph and $\{u_n\}\subset H_{\mu}^1(\mathcal{G})$ is a bounded Palais-Smale sequence of $E(\cdot\,,\cG)$ constrained on $H_{\mu}^1(\mathcal{G})$. Then there exists $u\in H^1(\mathcal{G})$ such that, up to a subsequence, $u_n \to u$ strongly in $H_{\mu}^1(\mathcal{G})$.
\end{lemma}

\begin{definition}\label{def: morse}
For any graph $\mathcal{F}$ (not necessarily compact) and any solution $U \in C(\mathcal{F}) \cap H^1_{{\rm loc}}(\mathcal{F})$, not necessarily in $H^1(\mathcal{F})$, of
\begin{equation}\label{eq U}
\begin{cases}
-U'' + \lambda U = \rho |U|^{p-2} U & \text{in $\mathcal{F}$}, \\
\sum_{{\rm e} \succ {\rm v}} U'({\rm v}) = 0 &  \text{for any vertex ${\rm v}$ of $\mathcal{F}$},
\end{cases}
\end{equation}
with $\lambda, \rho \in \R$, we consider
\begin{equation}\label{second differential}
Q(\varphi;U, \mathcal{F}):= \int_{\mathcal{F}} \left(|\varphi'|^2 + (\lambda-(p-1)\rho|U|^{p-2}) \varphi^2\right)\,dx, \quad \forall \varphi \in H^1(\mathcal{F}) \cap C_c(\mathcal{F}).
\end{equation}
The \emph{Morse index of $U$}, denoted by $m(U)$, is the maximal dimension of a subspace $W \subset H^1(\mathcal{F}) \cap C_c(\mathcal{F})$ such that $Q(\varphi; U, \mathcal{F})<0$ for all $\varphi \in W \setminus \{0\}$.
\end{definition}

Note that this is the definition of Morse index as solution to \eqref{eq U}, and not as critical point of the energy functional under the $L^2$ constraint (see Definition \ref{def: app morse} below).

\medskip

Lemma \ref{Embedding} is a useful result which exploits the compactness of the reference graph $\cG$. However, as already anticipated in the introduction, in the present setting even the existence of a \emph{bounded} Palais-Smale sequence at the mountain pass level is not straightforward. To overcome this issue, we introduce the family of functionals
\[
E_\rho(u, \cG) = \frac12 \int_{\cG} |u'|^2\,dx - \frac{\rho}p\int_{\cG} |u|^p\,dx,
\]
depending on the parameter $\rho \in [1/2,1]$. The idea is to adapt the monotonicity trick \cite{J-PRSE1999} on this family.

The main result of this section is the following:
\begin{proposition}\label{prop: ex for ae rho}
Let $\mu \in (0, \mu_1)$. For almost every $\rho \in [1/2,1]$, there exists a critical point $u_\rho$ of $E_\rho(\cdot,\cG)$ on $H^1_\mu(\cG)$, at level $c_{\rho}>E_\rho(\kappa_\mu, \cG)$, which solves
\begin{equation}\label{pb rho}
\begin{cases}
-u_\rho'' + \lambda_\rho u_\rho = \rho u_\rho^{p-1}, \quad u_\rho>0 & \text{in $\mathcal{\cG}$}, \\
\sum_{{\rm e} \succ {\rm v}} u_\rho'({\rm v}) = 0 &  \text{for any vertex ${\rm v}$},
\end{cases}
\end{equation}
for some $\lambda_\rho >0$. Moreover, its Morse index satisfies $m(u_\rho) \le 2$.
\end{proposition}

In the proof of the proposition, the value of $\mu \in (0, \mu_1)$ is fixed and will not change. As a first step, we show that the family of functionals $E_\rho(\cdot\,,\cG)$ has a mountain pass geometry on $H^1_\mu(\cG)$ around the constant local minimizer $\kappa_{\mu}$, uniformly with respect to $\rho$.

\begin{lemma}\label{Lem2-20-1}
There exists $w \in H^1_\mu(\cG)$ such that, setting
\[
\Gamma:= \left\{ \gamma \in C([0,1], H^1_\mu(\cG)): \ \gamma(0) = \kappa_\mu, \ \gamma(1) = w\right\},
\]
we have that
\[
c_\rho:= \inf_{\gamma \in \Gamma} \max_{t \in [0,1]} E_\rho(\gamma(t), \cG) > E_\rho(\kappa_\mu, \cG) = \max\{E_\rho(\gamma(0), \cG), E_\rho(\gamma(1), \cG) \}, \quad \forall \rho \in \left[\frac12,1\right].
\]
\end{lemma}

\begin{remark}
Note that the functions $\kappa_\mu$ and $w$, and hence also $\Gamma$, are independent of $\rho$.
\end{remark}

\begin{proof}
Since $\rho \leq 1$, and taking advantage of the monotonicity, we see from the proof of Proposition \ref{constant-solution} that $\kappa_{\mu}$ remains a strict local minimizer of $E_{\rho}(\cdot\,,\mathcal{G})$ in $H^1_{\mu}(\mathcal{G})$ for all $\rho\in[1/2, 1]$. \\
More precisely, for any $\rho\in[1/2, 1]$ there exists a ball $B(\kappa_{\mu}, r_\rho)$ of center $\kappa_{\mu}$ in $H^1_{\mu}(\mathcal{G})$ and radius $r_\rho>0$ such that $\kappa_\mu$ strictly minimizes $E_\rho(\cdot\,\cG)$ in $\overline{B(\kappa_\mu, r_\rho)}$, and
\begin{eqnarray}\label{3-1-1}
\inf\limits_{u\in \partial B(\kappa_{\mu}, r_\rho)}E_{\rho}(u, \mathcal{G})>E_{\rho}(\kappa_{\mu}, \mathcal{G}) > E_{1}(\kappa_{\mu}, \mathcal{G}).
\end{eqnarray}

Let $\mathrm e$ be any edge of $\cG$; we identify ${\mathrm e}$ with the interval $[-\ell_{\rm e}/2, \ell_{\rm e}/2]$. Then any compactly supported $H^1$ function $v$ on such interval, with mass $\mu$, can be seen as a function in $H^1_\mu(\cG)$. Denoting by $v_t(x) := t^{1/2} v(t x)$, with $t >1$, it is not difficult to check that $v_t \in H^1_\mu(\cG)$ (notice in particular that the support of $v_t$ is shrinking as $t$ becomes larger),
and that
\begin{align*}
E_\rho(v_t, \cG) &= \frac{t^2}{2} \int_{{\mathrm e}} |v'|^2\,dx- \frac{\rho t^{\frac{p-2}{2}}}{p}\int_{{\mathrm e}} |v|^p\,dx  \le  \frac{t^2}2\left( \int_{{\mathrm e}} |v'|^2\,dx - \frac{t^{\frac{p-6}{2}}}{p}\int_{{\mathrm e}} |v|^p\,dx\right),
\end{align*}
for every $\rho \in [1/2,1]$. Since $p>6$,
\[
E_\rho(v_t,\cG)<E_1(\kappa_\mu, \cG) < E_{\rho}(\kappa_{\mu}, \mathcal{G})
\]
for $t$ sufficiently large (independent of $\rho$). Taking now $w=v_t$ with any such choice of $t$ in the definition of $\Gamma$, the above estimate and the minimality of $\kappa_\mu$ in $\overline{B(\kappa_\mu, r_\rho)}$ imply that $w \not \in B(\kappa_\mu, r_\rho)$. Therefore, by continuity, for any $\gamma \in \Gamma$ there exist $t_\gamma \in [0,1]$ such that $\gamma(t_\gamma) \in \pa B(\kappa_\mu, r_\rho)$; and hence, by \eqref{3-1-1},
\[
\max_{t \in [0,1]} E_\rho(\gamma(t),\cG) \ge E_\rho(\gamma(t_\gamma),\cG) > \inf\limits_{u\in \partial B(\kappa_{\mu}, r_\rho)}E_{\rho}(u, \mathcal{G})>E_{\rho}(\kappa_{\mu}, \mathcal{G}) = \max\{ E_{\rho}(\kappa_{\mu}, \mathcal{G}), E_{\rho}(w, \mathcal{G})\},
\]
which completes the proof.
\end{proof}

At this point we wish to use the monotonicity trick on the family of functionals $E_\rho(\cdot\,\cG)$, in order to obtain a bounded Palais-Smale sequence at level $c_{\rho}$ for almost every $\rho \in [1/2,1]$. In fact, we need a stronger result carrying also a ``approximate Morse-index" information, Theorem \ref{thm: monot trick second order} below, proved in \cite{BoChJeSo}.

We recall the general setting in which the theorem is stated. Let $(E,\langle \cdot, \cdot \rangle)$ and $(H,(\cdot,\cdot))$ be two \emph{infinite-dimensional} Hilbert spaces and assume that:
\[ E\hookrightarrow H \hookrightarrow E',\]
with continuous injections.
 For simplicity, we assume that the continuous injection $E\hookrightarrow H$ has norm at most $1$ and identify $E$ with its image in $H$. We also introduce:  \[ \begin{cases} \|u\|^2=\langle u,u \rangle,\\ |u|^2=(u,u),\end{cases}\quad u\in E,\]
and, for $\mu \in (0,+\infty)$, we define \[ S_\mu= \{ u \in E, |u|^2=\mu \}. \]
For our application, it is plain that $E=H^1(\cG)$ and $H=L^2(\cG)$.

\begin{definition}
Let $\phi : E \rightarrow \mathbb{R}$ be a $C^2$-functional on $E$ and $\alpha \in (0,1]$. We say that $\phi'$ and $\phi''$ are $\alpha$-H\"older continuous on bounded sets if for any $R>0$ one can find $M=M(R)>0$ such that for any $u_1,u_2\in B(0,R)$:
\begin{equation}\label{Holder}
||\phi'(u_1)-\phi'(u_2)|| \leq M ||u_2-u_1||^{\alpha}, \quad ||\phi''(u_1)-\phi''(u_2)|| \leq M||u_1-u_2||^\alpha.
\end{equation}
\end{definition}

\begin{definition}\label{def D}
Let $\phi$ be a $C^2$-functional on $E$, for any $u\in E$ define the continuous bilinear map:
\[ D^2\phi(u)=\phi''(u) -\frac{\phi'(u)\cdot u}{|u|^2}(\cdot,\cdot).  \]
\end{definition}

\begin{remark}\label{rema:hessian}
If $u$ is a critical point of the functional $\phi|_{\M}$ then the restriction of $D^2\phi(u)$ to $T_u\M$ coincides with the constrained \emph{Hessian} of $\phi|_{\M}$ at $u$ (as introduced in Proposition \ref{constant-solution}.)
\end{remark}

\begin{definition}\label{def: app morse}
Let $\phi$ be a $C^2$-functional on $E$, for any $u\in \M$ and $\theta >0$, we define
 the \emph{approximate Morse index} by
\[
\tilde m_\theta(u) = \sup \left\{\dim\,L\left| \begin{array}{l} \ L \text{ is a subspace of $T_u \M$ such that: }
D^2|_{\M}\phi(u) (\varphi, \varphi) <-\theta \|\varphi\|^2, \quad \forall \varphi \in L \end{array}\right.\right\}.
\]
If $u$ is a critical point for the constrained functional $\phi|_{\M}$ and $\theta=0$, we say that this is the \emph{Morse index of $u$ as constrained critical point}.
\end{definition}

\begin{theorem}[Theorem 1 in \cite{BoChJeSo}]\label{thm: monot trick second order}
Let $I \subset (0,+\infty)$ be an interval and consider a family of $C^2$ functionals $\Phi_\rho: E \to \mathbb{R}$ of the form
\[
\Phi_\rho(u) = A(u) -\rho B(u), \qquad \rho \in I,
\]
where $B(u) \ge 0$ for every $u \in E$, and
\begin{equation}\label{hp coer}
\text{either $A(u) \to +\infty$~ or $B(u) \to +\infty$ ~ as $u \in E$ and $\|u\| \to +\infty$.}
\end{equation}
Suppose moreover that $\Phi_\rho'$ and $\Phi_\rho''$ are $\alpha$-H\"older continuous on bounded sets for some $\alpha \in (0,1]$.
Finally, suppose that there exist $w_1, w_2 \in \M$ (independent of $\rho$) such that, setting
\[
\Gamma= \left\{ \gamma \in C([0,1],\M): \ \gamma(0) = w_1, \quad \gamma(1) = w_2\right\},
\]
we have
\begin{equation}\label{mp geom}
c_\rho:= \inf_{\gamma \in \Gamma} \ \max_{ t \in [0,1]} \Phi_\rho(\gamma(t)) > \max\{\Phi_\rho(w_1), \Phi_\rho(w_2)\}, \quad \rho \in I.
\end{equation}
Then, for almost every $\rho \in I$, there exist sequences $\{u_n\} \subset \M$ and $\zeta_n \to 0^+$ such that, as $n \to + \infty$,
\begin{itemize}
\item[(i)] $\Phi_\rho(u_n) \to c_\rho$;
\item[(ii)] $||\Phi'_\rho|_{\M}(u_n)|| \to 0$;
\item[(iii)] $\{u_n\}$ is bounded in $E$;
\item[(iv)] $\tilde m_{\zeta_n}(u_n) \le 1$.
\end{itemize}
\end{theorem}

We are ready to give the proof of Proposition \ref{prop: ex for ae rho}.

\begin{proof}[Proof of Proposition \ref{prop: ex for ae rho}]
We apply Theorem \ref{thm: monot trick second order} to the family of functionals $E_\rho(\cdot\,, \cG)$, with $E=H^1(\cG)$, $H=L^2(\cG)$, $S_\mu = H^1_\mu(\cG)$, and $\Gamma$ defined in Lemma \ref{Lem2-20-1}. Setting
\[
A(u) = \frac12\int_{\cG} |u'|^2\,dx \quad \text{and} \quad  B(u) = \frac{\rho}{p}\int_{\cG} |u|^p.
\]
assumption \eqref{hp coer} holds, since we have that
\[
u \in H^1_\mu(\cG), \ \|u\| \to +\infty \quad \implies \quad A(u) \to +\infty.
\]
Moreover, assumption \eqref{Holder} holds since the unconstrained first and second derivatives of $E_\rho$ are of class $C^1$, and hence locally H\"older continuous, on $H_\mu^1(\cG)$.

In this way, for almost every $\rho \in [1/2,1]$ there exist a bounded Palais-Smale sequence $\{u_n\}$ for the constrained functional $E_\rho(\cdot\,,\cG)|_{H^1_\mu(\cG)}$ at level $c_\rho$, and $\zeta_n \to 0^+$, such that $\tilde m_{\zeta_n}(u_n) \le 1$. Moreover, as explained in \cite[Remark 1.4]{BoChJeSo}, since $u \in S_\mu \ \mapsto \ |u| \in S_\mu$, $w_1, w_2 \ge 0$, the map $u \mapsto |u|$ is continuous, and $E_\rho(u, \cG) = E_\rho(|u|, \cG)$, it is possible to choose $\{u_n\}$ with the property that $u_n \ge 0$ on $\cG$. By Lemma \ref{Embedding}, we have that $u_n \to u_\rho$ strongly in $H^1(\cG)$, and $u_\rho \ge 0$ is a constrained critical point, thus a non-negative solution to \eqref{pb rho}, for $\lambda_\rho=\lambda(u_\rho)$ (Lemma \ref{Embedding} is stated for the particular value $\rho=1$; however, it is immediate to check that this choice does not play any role in the proof). The case when $u_\rho$ vanishes in one (or more) vertexes can be easily ruled out by the Kirchhoff condition, the uniqueness theorem for ODEs, and the fact that $u_\rho \ge 0$. Thus, $u_\rho$ is strictly positive on each vertex, whence $u_\rho >0$ in $\cG$ by the strong maximum principle. \smallskip

It remains to show that the Morse index $m(u_\rho)$, defined in Definition \ref{def: morse} with $\lambda = \lambda(u_\rho)$ is at most $2$. This result can be directly deduce from \cite[Theorem 3]{BoChJeSo} but we prove it here in our setting for completeness. We omit the dependence of the functionals $E_\rho(\cdot\,,\cG)$ on $\cG$, to simplify the notation. Defining
$$\overline{\lambda}_{\rho}:=  -\frac{1}{\mu}E'_\rho(u_{\rho})\cdot u_{\rho} =  - \lim_{n \to \infty}\frac{1}{\mu}E'_\rho(u_n)\cdot u_n,$$
we conclude from Theorem \ref{thm: monot trick second order} (ii) that $\overline{\lambda}_{\rho} = \lambda_{\rho}$, we refer to \cite[Remark 1.2]{BoChJeSo} for more detail.\smallskip

To show that $u_{\rho} \in \M$  has Morse index at most $1$ as constrained critical point, see Definition \ref{def: app morse}, we assume by contradiction that there exists a $W_0 \subset T_{u}\M$ with $\dim W_0 =2$ such that
$$D^2 E_{\rho}(u_{\rho})(w,w)  < 0, \quad \mbox{for all } w \in W_0 \backslash \{0\}. $$
Since $W_0$ is of finite dimension, by compactness and homogeneity, there exists a $\beta >0$ such that
\begin{equation*}
D^2 E_{\rho}(u_{\rho})(w,w) < - \beta ||w||^2,  \quad \mbox{for all } w \in W_0.
\end{equation*}
Now, from \cite[Corollary 1]{BoChJeSo} or using directly that $E_\rho'$ and $E_\rho''$ are $\alpha$-H\"older continuous on bounded sets for some $\alpha \in (0,1]$, we deduce that there exists a $\delta_1 >0$ such that, for any $v \in \M$ such that $||v-u || \leq \delta_1$,
\begin{equation}\label{L-conditionD2step1} D^2E_{\rho}(v)(w,w) < - \frac{\beta}{2} ||w||^2 \quad \mbox{for all } w \in W_0.\end{equation}
Since $\{u_n\} \subset \M $ converges to $u$ we have that $||u_n -u|| \leq \delta_1$ for $n \in \N$ large enough. Then since $\dim W_0 >1$,
\eqref{L-conditionD2step1} provides a contradiction with Theorem \ref{thm: monot trick second order} (iv) where we recall that $\zeta_n \to 0^+$.  Finally, recalling that $\M$ is of codimension 1 in $H^1(\cG)$ and observing that, for any $ w \in H^1(\cG)$,
$$D^2 E_{\rho}(u_{\rho})(w,w) := E''_{\rho}(u)(w,w) +\lambda_{\rho} (w,w) = \int_{\cG} \left[ |w'|^2 + \left( \lambda_{\rho} -(p-1) |u_{\rho}|^{p-2}\right) w^2\right]  \, dx,$$
we obtain that $m(u_{\rho}) \leq 2$.
\end{proof}

\section{Blow-up Phenomena}\label{sec: blow-up}

Proposition \ref{prop: ex for ae rho} does not ensure the existence of a mountain pass solution for the original problem obtained when $\rho=1$. However, it gives the existence of a sequence $\rho_n \to 1^-$, with a corresponding sequence of mountain pass critical points $u_{\rho_n} \in H^1_\mu(\cG)$ of $E_{\rho_n}(\cdot\,, \cG)$, constrained on $H^1_\mu(\cG)$. We aim to show that $\{u_{\rho_n}\}$ converges to a constrained critical point of $E_1(\cdot\,,\cG)$. To this purpose, it is sufficient to prove that $\{u_{\rho_n}\}$ is bounded in $H^1(\cG)$, thanks to Lemma \ref{Embedding}. The advantage of working with $\{u_{\rho_n}\}$ is that this is a sequence of \emph{solutions of approximating problems with uniformly bounded Morse index}. In this section we perform a blow-up analysis for this type of sequences, in the spirit of \cite{EspPet}. This analysis, of independent interest, will be used in the next section to gain the desired boundedness of $\{u_{\rho_n}\}$.

A somehow related study, regarding least action solutions, was previously performed in \cite{DGMP}.

\subsection*{General setting for the blow-up analysis.}

For the sake of generality, in what follows we consider a general metric graph satisfying the following assumption:
\[
\cG \text{ has a finite number of vertexes and edges (but is not necessarily compact).}
\]
Let $\{u_n\} \in H^1(\cG)$ be a sequence of positive solutions of the NLS equation, coupled with Kirchhoff condition at the vertexes:
\begin{equation}\label{eq: blow-up}
\begin{cases}
-u_n''+\lambda_n u_n = \rho_n u_n^{p-1} & \text{on $\cG$}, \\
u_n>0 & \text{on $\cG$}, \\
\sum_{\rm{e} \succ \rm{v}} u_{e,n}'(\rm{v}) = 0& \forall \rm{v} \in \mathcal{V},
\end{cases}
\end{equation}
where $\rho_n \to 1$ (in fact, it would be sufficient to ask that $\rho_n \to \rho>0$, regardless of the value of $\rho$), and $\lambda_n \in \R$.
%

We denote by $B_r(x_0)=\{x \in \cG: \ {\rm dist}(x,x_0)<r\}$. Moreover, we denote by $\cG_m$ the star-graph with $m \ge 1$ half-lines glued together at their common origin $0$ (note that $\cG_1=\R^+$, and $\cG_2$ is isometric to $\R$).

It is also convenient to recall the definition of $Q(\phi; u, \mathcal{G})$, see \eqref{second differential}.

At first, we note that if $\lambda_n \to +\infty$, then $u_n$ blows-up along any sequence of local maximum points.

\begin{lemma}\label{lem: stima max}
Let $x_n \in \cG$ be a local maximum point for $u_n$. Then
\[
u_n(x_n) \ge \lambda_n^\frac{1}{p-2}.
\]
\end{lemma}

\begin{proof}
Let $\rm{e}$ be an edge of $\cG$ such that $x_n \in \rm{e} \simeq [0,\ell_e]$; it is plain that $u_n|_e \in C^2([0,\ell_e])$, by regularity. If $x_n$ is in the interior of $e$, then $u_n''(x_n) \le 0$; if instead $x_n$ is a vertex of $e$, then, by the Kirchhoff condition, $u_n'(x_n)$ must vanish, and hence $u_n''(x_n) \le 0$ again. In both cases, the equation of $u_n$ (which holds on the whole closed interval $[0,\ell_e]$) yields
\[
\lambda_n u_n(x_n) - \rho_n u_n^{p-1}(x_n) = u_n''(x_n) \le 0,
\]
whence the thesis follows.
\end{proof}


The next theorem provides a precise behavior, close to a local maximum point, of the sequence $\{u_n\}$, as $\lambda_n \to +\infty$ while $m(u_n)$ remains bounded. In the statement and in the proof, we will systematically identify an edge $\edge$ with the interval $[0,\ell_{\edge}]$, where $\ell_{\edge}$ denotes the length of $\edge$. Since in this section we allow $\cG$ to be non-compact, it is admissible that $\ell_{\edge}=+\infty$ (clearly, in such case $\edge \simeq [0,+\infty)$; unless it is necessary, we will not distinguish these cases).

\begin{theorem}\label{thm: blow-up 1}
Suppose that
\[
\lambda_n \to +\infty \quad \text{ and } \quad m(u_n) \le \bar k \quad \text{ for some $\bar k \ge 1$}.
\]
Let $x_n \in \cG$ be such that, for some $R_n \to \infty$,
\begin{equation}\label{hp max}
u_n(x_n) = \max_{B_{R_n \tilde \eps_n}(x_n)} u_n \quad \text{where } \tilde \eps_n=(u_n(x_n))^{-\frac{p-2}{2}} \to 0.
\end{equation}
Suppose moreover that
\begin{equation}\label{int max}
\limsup_{n \to \infty} \frac{{\rm{dist}}(x_n, \mathcal{V})}{\tilde \eps_n} = +\infty.
\end{equation}
Then, up to a subsequence, the following holds:
\begin{itemize}
\item[($i$)] all the $x_n$ lie in the interior of the same edge ${\rm{e}} \simeq [0,\ell_e]$.
\item[($ii$)] Setting $\eps_n = \lambda_n^{-\frac12}$, we have that
\begin{equation}\label{rel eps tilde eps}
\begin{split}
\frac{\tilde \eps_n}{\eps_n} &\to (0,1], \\
\frac{{\rm{dist}}(x_n, \mathcal{V})}{\eps_n} &\to +\infty \quad \text{as $n \to \infty$},
\end{split}
\end{equation}
and the scaled sequence
\begin{equation}\label{bu int}
v_n(y):= \eps_n^{\frac{2}{p-2}} u_n(x_n + \eps_n y) \quad \text{for }y \in \frac{[0,\ell_e] -x_n}{\eps_n}
\end{equation}
converges to $V$ in $C^2_{\rm{loc}}(\R)$ as $n \to \infty$, where $V \in H^1(\R)$ is the (unique) positive finite energy solution to
\[
\begin{cases}
-V'' + V = V^{p-1}, \quad U_0>0 & \text{in $\R$}, \\
V(0) = \max_\R V,\\
V(x) \to 0 & \text{as $|x| \to +\infty$}.
\end{cases}
\]
\item[($iii$)] There exists $\phi_n \in C^\infty_c(\cG)$, with ${\rm{supp}}\, \phi_n \subset B_{\bar R \eps_n}(x_n)$ for some $\bar R>0$, such that
\[
Q(\phi_n; u_n, \cG) <0.
\]
\item[($iv$)] For all $R>0$ and $q \ge 1$, we have that
\[
\lim_{n \to \infty} \lambda_n^{\frac12- \frac{q}{p-2}} \int_{B_{R \eps_n}(x_n)} u_n^{q}\,dx=  \lim_{n \to \infty} \int_{B_R(0)} v_n^q\,dy = \int_{B_R(0)} V^q\,dy.
\]
\end{itemize}
If, instead of \eqref{int max}, we suppose that
\begin{equation}\label{ver max}
\limsup_{n \to \infty} \frac{{\rm{dist}}(x_n, \mathcal{V})}{\tilde \eps_n} <+\infty,
\end{equation}
then, up to a subsequence,
\begin{itemize}
\item[($i'$)] $x_n \to {\rm{v}} \in \mathcal{V}$, and all the $x_n$ lie on the same edge ${\rm{e}}_1 \simeq [0,\ell_1]$, where the vertex ${\rm{v}}$ is identified by the coordinate $0$ on ${\rm{e}}_1$.
\item[($ii'$)] Let ${\rm{e}}_2 \simeq [0,\ell_2]$, \dots, ${\rm{e}}_m \simeq [0,\ell_m]$ be the other edges of $\cG$ having ${\rm{v}}$ as a vertex (if any), where ${\rm{v}}$ is identified by the coordinate $0$ on each ${\rm{e}}_i$. Setting $\eps_n = \lambda_n^{-\frac12}$, we have that
\begin{equation}\label{rel eps tilde eps 2}
\begin{split}
\frac{\tilde \eps_n}{\eps_n} &\to (0,1], \\
\limsup_{n \to \infty} &\frac{{\rm{dist}}(x_n, \mathcal{V})}{\eps_n} < +\infty,
\end{split}
\end{equation}
and the scaled sequence defined by
\[
v_n(y):= \eps_n^{\frac{2}{p-2}} u_n(\eps_n y) \quad \text{for }y \in \frac{{\rm{e}}_i}{\eps_n}, \text{ for $i=1,\dots,m$},
\]
converges to a limit $V$ in $C^0_{\rm{loc}}(\cG_m)$ as $n \to \infty$. Denoting by $V_i$ the restriction of $V$ to the $i$-th half-line $\ell_i$ of $\cG_m$, and by $v_{i,n}$ the restriction of $v_n$ to ${\rm e}_i/\eps_n$, we have moreover that $v_{i,n} \to V_i$ in $C^2_{\rm{loc}}([0,+\infty))$. Finally, $V \in H^1(\cG_m)$ is a positive finite energy solution to the NLS equation on the star-graph
\[
\begin{cases}
-V'' + V = V^{p-1}, \quad V>0 & \text{in $\cG_m$},\\
\sum_{i=1}^m V_i'(0^+) =0, \\
V(x) \to 0 & \text{as }\dist(x,0) \to \infty
\end{cases}
\]
with a global maximum point $\bar x$ located on $\ell_1$, whose coordinate is
\[
\bar x = \lim_{n \to \infty} \bar x_n \in [0,+\infty), \quad \text{where} \quad \bar x_n := \frac{{\rm{dist}}(x_n, \mathcal{V})}{\eps_n} .
\]
\item[($iii'$)] There exists $\phi_n \in C^\infty_c(\cG)$, with ${\rm{supp}}\,\phi_n \subset B_{\bar R \eps_n}(x_n)$ for some $\bar R>0$, such that
\[
Q(\phi_n; u_n, \cG) <0.
\]
\item[($iv'$)] For all $R>0$ and $q \ge 1$, we have that
\[
\lim_{n \to \infty} \lambda_n^{\frac12- \frac{q}{p-2}} \int_{B_{R \eps_n}(x_n)} u_n^{q}\,dx =   \lim_{n \to \infty} \int_{B_R(\bar x_n)} v_n^q\,dy = \int_{[0, \bar x + R]} V_1^q\,dy + \sum_{i=2}^m \int_{[0, R- \bar x]} V_i^q\,dy  = \int_{B_R(\bar x)} V^q\,dy
\]
(where $B_R(\bar x_n)$ and $B_R(\bar x)$ denote the balls in the scaled and in the limit graphs, respectively).
\end{itemize}
\end{theorem}

The proof of the theorem is divided into several intermediate steps. We start with some preliminary results.

%
%
%
%
\begin{lemma}\label{lem: dec fmi}
Let $U \in H^1_{{\rm loc}}(\cG_m)$ be a solution to
\begin{equation}\label{nls on star}
\begin{cases}
-U'' + \lambda U = \rho U^{p-1} & \text{in $\cG_m$}, \\
U>0 & \text{in $\cG_m$}, \\
\sum_{i=1}^m U_i'(0) = 0,
\end{cases}
\end{equation}
for some $p>2$, $\rho, \lambda>0$, where $U_i$ denotes the restriction of $U$ on the $i$-th half-line of $\cG_m$. Suppose that $U$ is stable outside a compact set $K$, in the sense that $Q(\varphi; U, \cG_m) \ge 0$ for all $\varphi \in H^1(\cG_m) \cap C_c(\cG_m \setminus K)$. Then $U(x) \to 0$ as ${\rm{dist}}(x,0) \to +\infty$, and $U \in H^1(\cG)$.
\end{lemma}

The proof is analogue to the one of \cite[Theorem 2.3]{EspPet}, and hence we omit it.

\begin{remark}
Clearly, by the density of $H^1([0,+\infty)) \cap C_c([0,+\infty))$ in $H^1([0,+\infty))$, any solution with finite Morse index is stable outside a compact set.
\end{remark}

\begin{lemma}\label{lem: pos mi}
Let $U \in H^1(\cG_m)$ be any non-trivial  solution of \eqref{nls on star}. Then its Morse index $m(U)$ is strictly positive.
\end{lemma}
\begin{proof}
Thanks to the Kirchhoff condition, it is not difficult to check that
\[
\int_{\cG_m} \left(|U'|^2 + \lambda U^2\right) dx = \int_{\cG_m} \rho |U|^p\, dx.
\]
Therefore
\[
Q(U; U, \cG_m) = (2-p) \int_{\cG_m} |U|^p\,dx<0,
\]
and the thesis follows by density of $H^1(\cG_m) \cap C_c(\cG_m)$ in $H^1(\cG_m)$.
\end{proof}

\begin{proof}[Proof of Theorem \ref{thm: blow-up 1} under assumption \eqref{int max}] This case is simpler than the one when \eqref{ver max} holds, since, roughly speaking, after rescaling we do not see the vertexes of $\cG$, and we obtain a limit problem on the line. We present in any case the proof for the sake of completeness. Since $\cG$ has a finite number of edges, up to a subsequence all the points $x_n$ belong to same edge ${\rm e}$, and (i) holds. Let $\tilde u_n$ be defined by
\[
\tilde u_n(y) := \tilde \eps_n^{\frac{2}{p-2}} u_n(x_n + \tilde \eps_n y) \quad \text{for }y \in \tilde {\rm e}_n:= \frac{{\rm e}-x_n}{\tilde \eps_n}.
\]
Notice that any interval $[-a,a]$, with $a>0$, is contained in $\tilde {\rm e}_n$ for sufficiently large $n$. Indeed, $({\rm e}-x_n)/\tilde \eps_n$ contains the set
\[
\{y \in \R: \ |\tilde \eps_n y|< {\rm dist}(x_n, \mathcal{V})\} = \left\{y \in \R: \ |y|< \frac{ {\rm dist}(x_n, \mathcal{V})}{\tilde \eps_n}\right\},
\]
which exhausts the whole line $\R$ as $n \to \infty$, by \eqref{int max}.

Now, on every compact $[-a,a]$ we have that $\tilde u_n(0) = 1 = \max_{[-a,a]} \tilde u_n$ for $n$ large (since $u_n(x_n) = \max_{B_{R_n \tilde \eps_n}(x_n)} u_n$ for some $R_n \to +\infty$), and
\[
-\tilde u_n'' + \tilde \eps_n^2 \lambda_n \tilde u_n = \rho_n \tilde u_n^{p-1}, \quad \tilde u_n >0 \quad \text{in $\tilde {\rm{e}}_n$}.
\]
Furthermore, by Lemma \ref{lem: stima max}
\[
\tilde \eps_n^2 \lambda_n \in (0,1], \quad \forall n.
\]
Thus, by elliptic estimates, we have that $\tilde u_n \to \tilde u$ in $C^2_{\rm loc}(\R)$, and the limit $\tilde u$ solves
\begin{equation}\label{limit pb}
- \tilde u''+\tilde \lambda \tilde u = \tilde u^{p-1}, \quad \tilde u  \ge 0 \quad \text{in $\R$}
\end{equation}
for some $\tilde \lambda \in [0,1]$. By local uniform convergence, $\tilde u(0) = 1$, and hence $\tilde u>0$ in $\R$ by the strong maximum principle. We claim that
\begin{equation}\label{cl: fin morse}
\text{the Morse index of $\tilde u$ is bounded by $\bar k$}.
\end{equation}
If by contradiction this is false, then there exists $k> \bar k$ functions $\phi_1,\dots,\phi_k \in H^1(\R) \cap C_c(\R)$, linearly independent in $H^1(\R)$, 
such that $Q(\phi_i; \tilde u, \R)<0$ for every $i \in \{1, \cdots, k\}$. Let then
\[
\phi_{i,n}(x)= \tilde \eps_n^{\frac{1}{2}} \phi_i\left(\frac{x-x_n}{\tilde \eps_n}\right)
\]
Since $\phi_i$ has compact support, the functions $\phi_{i,n}$ can be regarded as functions in $H^1({\rm e})$, and hence in $H^1(\cG)$, for every $n$ large, thanks to \eqref{int max}. Indeed, if ${\rm supp} \, \phi_i \subset [-M,M]$, then
\begin{align*}
\left\{x \in \R: \ \frac{x-x_n}{\tilde \eps_n} \subset [-M,M]\right\} &= [x_n-\tilde \eps M, x_n + \tilde \eps_n M] \\
&\subset \left[x_n-\tilde \eps_n \frac{ {\rm dist}(x_n, \mathcal{V})}{\tilde \eps_n}, x_n + \tilde \eps_n \frac{ {\rm dist}(x_n, \mathcal{V})}{\tilde \eps_n}\right] \subset {\rm e}.
\end{align*}
Moreover, $\phi_{1,n},\dots,\phi_{k,n}$ are linearly independent in $H^1(\cG)$, and, by scaling,
\[
Q(\phi_{i,n}; u_n, \cG) = Q(\phi_{i,n}; u_n, {\rm e}) = Q(\phi_i; \tilde u_n, \tilde {\rm e}_n) \to Q(\phi_i; \tilde u, \R)<0.
\]
This implies that $m(u_n) \ge k> \bar k$ for sufficiently large $n$, a contradiction. Therefore, claim \eqref{cl: fin morse} is proved.
%
%
To sum up, $\tilde u$ is a finite Morse index non-trivial solution to \eqref{limit pb}, for some $\tilde \lambda \in [0,1]$. Having $\tilde \lambda=0$ is however not possible, since by phase plane analysis the equation $\tilde u'' + \tilde u^{p-1}=0$ in $\R$ has only periodic sign-changing solution, but the trivial one. Now, by Lemma \ref{lem: dec fmi}, $\tilde u \to 0$ as $|x| \to +\infty$, and $\tilde u \in H^1(\R)$. Therefore,
\begin{equation}\label{rel max lam 1}
0<\liminf_{n \to \infty} \frac{\lambda_n}{(u_n(x_n))^{p-2}} \le \limsup_{n \to \infty} \frac{\lambda_n}{(u_n(x_n))^{p-2}} \le 1,
\end{equation}
which proves the first estimate in \eqref{rel eps tilde eps}. At this point it is equivalent, but more convenient, to work with $v_n$ defined by \eqref{bu int} rather than with $\tilde u_n$. By \eqref{int max} and \eqref{rel max lam 1},
\[
\limsup_{n \to \infty} \frac{{\rm dist}(x_n,\mathcal{V})}{\eps_n} = +\infty.
\]
Thus, similarly as done before, one can show that $v_n$ converges to a limit function $v$ in $C^2_{\rm loc}(\R)$, such that
\[
- v''+v = v^{p-1} \quad v  \ge 0 \quad \text{in $\R$};
\]
moreover, $v$ has a positive global maximum $v(0) \ge 1$ (thus $v>0$ in $\R$), has finite Morse index $m(v) \le \bar k$, and hence, by Lemma \ref{lem: dec fmi}, $v \to 0$ as $|x| \to \infty$, and $v \in H^1(\R)$. It is well known that there exists only one such solution, denoted by $V$. Thus, (ii) is proved. Point (iv) follows directly by local uniform convergence. Finally, point (iii) is a consequence of the fact that the Morse index of $V$ is positive (see Lemma \ref{lem: pos mi}; in fact, it is well known that in fact $m(V)$ is precisely equal to $1$). This implies that there exists $\phi \in C^1_c(\R)$ such that $Q(\phi; V,\R)<0$; thus, defining
\[
\phi_{i,n}(x)= \eps_n^{\frac{1}{2}} \phi_i\left(\frac{x-x_n}{\eps_n}\right),
\]
we deduce that for sufficiently large $n$ we have $Q(\phi_{i,n};u_n, \cG)<0$, and ${\rm supp}\, \phi_{i,n} \subset B_{\bar R \eps_n}(x_n)$ for some $\bar R>0$.
\end{proof}

\begin{proof}[Proof of Theorem \ref{thm: blow-up 1} under assumption \eqref{ver max}]
Since $\tilde \eps_n \to 0$ and $\cG$ has a finite number of vertexes and edges, up to a subsequence the maximum points $x_n$ converge to a vertex ${\rm v}$, and belong to same edge ${\rm e}_1 \simeq [0, \ell_1]$; thus, ($i'$) holds, and we can suppose that
\[
\frac{d_n}{\tilde \eps_n} \to \eta \in [0,+\infty), \quad d_n:={\rm dist}(x_n,\mathcal{V}) = x_n.
\]
Let
\[
\tilde u_n(y):= \tilde \eps_n^{\frac{2}{p-2}} u_n(\tilde \eps_n y) \quad \text{for }y \in \tilde {\rm e}_{i,n}:= \frac{{\rm e}_i}{\tilde \eps_n}, \ \text{for $i=1,\dots,m$}.
\]
Note that $\tilde u_n$ is defined on a graph $\cG_{m,n}$ consisting in $m$ expanding edges, glued together at their common origin, which is identified with the coordinate $0$ on each edge $\tilde e_{i,n}$. In the limit $n \to \infty$, this graph converges to the star-graph $\cG_m$. Plainly, for every $a>\eta +1$ and large $n$
\[
\tilde u_n\left(\frac{x_n}{\tilde \eps_n}\right) = 1 = \max_{B_a(0)} \tilde u_n
\]
(since $u_n(x_n) = \max_{B_{R_n \tilde \eps_n}(x_n)} u_n$ for some $R_n \to +\infty$),
\[
-\tilde u_n'' + \tilde \eps_n^2 \lambda_n \tilde u_n = \rho_n \tilde u_n^{p-1}, \quad \tilde u_n >0
\]
on any edge of $\cG_{m,n}$, and the Kirchhoff condition at the origin holds. Also, by Lemma \ref{lem: stima max},
\[
\tilde \eps_n^2 \lambda_n \in (0,1] \quad \forall n.
\]
Thus, by elliptic estimates, we have that $\tilde u_n|_{\tilde {\rm e}_{i,n}}=: \tilde u_{i,n} \to \tilde u_i$ in $C^2_{\rm loc}([0,+\infty))$ for every $i$, and the limit $\tilde u_i$ solves
\begin{equation}\label{limit pb 2}
- \tilde u''_i+\tilde \lambda \tilde u_i = \tilde u_i^{p-1}, \quad \tilde u_i  \ge 0 \quad \text{in $(0,+\infty)$}
\end{equation}
for some $\tilde \lambda \in [0,1]$. Moreover, since $\tilde u_{n}$ is continuous on $\cG_{m,n}$ and by uniform convergence, $\tilde u_i(0) = \tilde u_j(0)$ for every $i \neq j$, so that $\tilde u \simeq (\tilde u_1,\dots, \tilde u_m)$ can be regarded as a function defined on $\cG_m$. Since the convergence $\tilde u_{i,n} \to \tilde u_i$ takes place in $C^2$ up to the origin, also the Kirchhoff condition passes to the limit. Now we exclude the case that $\tilde u \equiv 0$ on some half-line of $\cG_m$. By local uniform convergence, we have that
\[
\tilde u_1(\rho) = \lim_{n \to \infty} \tilde u_{1,n}\left(\frac{d_n}{\tilde \eps_n}\right) = 1.
\]
This implies that $\tilde u_1>0$ in $(0,+\infty)$, by the strong maximum principle. In turn, the Kirchhoff condition, the uniqueness theorem for ODEs, and the strong maximum principle again, ensure that $\tilde u_i>0$ on $(0,+\infty)$ for every $i$. Finally, we claim that
\begin{equation}\label{cl: fin morse 2}
\text{the Morse index of $\tilde u$ is bounded by $\bar k$}.
\end{equation}
The proof of this claim is completely analogue to the one of \eqref{cl: fin morse}. If by contradiction this is false, then there exists $k> \bar k$ functions $\phi_1,\dots,\phi_k \in H^1(\cG_m) \cap C_c(\cG_m)$, linearly independent in $H^1(\cG_m)$, such that $Q(\phi_i; \tilde u, \cG_m)<0$ for every $i \in \{1, \cdots, k\}$. Let then
\[
\phi_{i,n}(x)= \tilde \eps_n^{\frac{1}{2}} \phi_i\left(\frac{x}{\tilde \eps_n}\right).
\]
Since $\phi_i$ has compact support, the functions $\phi_{i,n}$ can be regarded as functions in $H^1(\cG)\cap C_c(\cG)$ for every $n$ large; precisely, ${\rm supp}(\phi_{i,n}) \subset B_{R \tilde \eps_n}(x_n)$ for some $R>2\rho$. Moreover, $\phi_{1,n},\dots,\phi_{k,n}$ are linearly independent in $H^1(\cG_m)$ and, by scaling,
\[
Q(\phi_{i,n}; u_n, \cG) = Q(\phi_{i,n}; u_n, {\rm e}) = Q(\phi_i; \tilde u_n, \tilde {\rm e}_n) \to Q(\phi_i; \tilde u, \R)<0.
\]
This implies that $m(u_n) \ge k> \bar k$ for sufficiently large $n$, a contradiction. Therefore, claim \eqref{cl: fin morse 2} is proved.

To sum up, $\tilde u$ is a finite Morse index non-trivial solution to \eqref{limit pb 2}, for some $\tilde \lambda \in [0,1]$. As before, the case $\tilde \lambda=0$ can be ruled out by phase-plane analysis, and hence, by Lemma \ref{lem: dec fmi}, $\tilde u \to 0$ as $|x| \to +\infty$, and $\tilde u \in H^1(\cG_m)$. Therefore,
\begin{equation}\label{rel max lam}
0<\liminf_{n \to \infty} \frac{\lambda_n}{(u_n(x_n))^{p-2}} \le \limsup_{n \to \infty} \frac{\lambda_n}{(u_n(x_n))^{p-2}} \le 1,
\end{equation}
which proves the first estimate in \eqref{rel eps tilde eps 2}. At this point it is equivalent, but more convenient, to work with $v_n$ defined in point ($ii'$) of the theorem, rather than with $\tilde u_n$. By \eqref{ver max} and \eqref{rel max lam},
\[
\limsup_{n \to \infty} \frac{{\rm dist}(x_n,\mathcal{V})}{\eps_n} <+\infty.
\]
Thus, similarly as done before, one can show that $v_n$ converges, in $C^0_{\rm loc}(\cG_m)$ and in $C^2_{{\rm loc}}([0,+\infty))$ on every half-line, to a limit function $V\simeq(V_1,\dots,V_m)$, which solves
\begin{equation}\label{limit pb v 2}
\begin{cases}
- V''+V = V^{p-1}, \quad V  \ge 0 & \text{in $\cG_m$}, \\
\sum_{i=1}^m V_i'(0^+) = 0;
\end{cases}
\end{equation}
furthermore, $V$ has a positive global maximum on the half-line $\ell_1$, $V_1(\bar x) \ge 1$ (thus $V>0$ in $\cG_m$), and has finite Morse index $m(V) \le \bar k$. Moreover, by Lemma \ref{lem: dec fmi}, $V \to 0$ as $|x| \to \infty$. Thus, ($ii'$) is proved. Point ($iv'$) follows directly by local uniform convergence. Finally, point ($iii'$) is a consequence of Lemma \ref{lem: pos mi}. This implies that there exists $\phi \in H^1(\cG_m) \cap C_c(\cG_m)$ such that $Q(\phi; V,\cG_m)<0$; thus, defining
\[
\phi_{i,n}(x)= \eps_n^{\frac{1}{2}} \phi_i\left(\frac{x-x_n}{\eps_n}\right),
\]
it is not difficult to deduce that for sufficiently large $n$ we have $Q(\phi_{i,n};u_n, \cG)<0$, and ${\rm supp}\, \phi_{i,n} \subset B_{\bar R \eps_n}(x_n)$ for some positive $\bar R$.
\end{proof}

Theorem \ref{thm: blow-up 1} allows to describe the pointwise blow-up behavior close to local maximum points. In what follows, we focus on the global behavior, and, in particular, on what happens far away from local maxima.

\begin{theorem}\label{thm: blow-up 2}
Let $\{u_n\} \subset H^1(\cG)$ be a sequence of solutions to \eqref{eq: blow-up} such that $\lambda_n \to +\infty$ and $m(u_n) \le \bar k$ for some $\bar k \ge 1$. There exist $k \in \{1,\dots,\bar k\}$, and sequences of points $\{P_n^1\}$, \dots, $\{P_n^k\}$, such that
\begin{gather}
\lambda_n \dist(P_n^i, P_n^j) \to +\infty, \quad \forall i \neq j, \label{36ep} \\
u_n(P_n^i) = \max_{B_{R_n \lambda_n^{-1/2}}(P_n^i)} u_n \quad \text{for some $R_n \to +\infty$, for every $i$,} \label{37ep}
\end{gather}
and constants $C_1, C_2>0$ such that
\begin{equation}\label{estim far away}
u_n(x) \le C_1 \lambda_n^{\frac{1}{p-2}} \sum_{i=1}^k e^{-C_2 \lambda_n^{\frac{1}{2}} \dist(x, P_n^i)}+C_1 \lambda_n^{\frac{1}{p-2}}\sum_{j=1}^{h} e^{-C_2 \lambda_n^{\frac{1}{2}} \dist(x, {\rm v}_j )},   \quad \forall x \in \cG \setminus \bigcup_{i=1}^k B_{R \lambda_n^{-1/2}} (P_n^i),
\end{equation}
where ${\rm v}_1, \dots, {\rm v}_h$ are all the vertexes of $\cG$.
%
\end{theorem}

\begin{proof}
The proof follows closely the one of \cite[Theorem 3.2]{EspPet}, and is divided into two steps.

\smallskip

\emph{Step 1) There exist $k \in \{1,\dots,\bar k\}$, and sequences of points $\{P_n^1\}$, \dots, $\{P_n^k\}$, such that \eqref{36ep} and \eqref{37ep} hold, and moreover
\begin{equation}\label{39ep}
\lim_{R \to +\infty} \left( \limsup_{n \to \infty} \  \lambda_n^{-\frac{1}{p-2}} \max_{d_n(x) \ge R \lambda_n^{-1/2}} u_n(x) \right) = 0.
\end{equation}
where $d_n(x) =\min\{ \dist(x,P_n^i): \ i=1,\dots,k\}$ is the distance function from $\{P_1^n, \dots, P_k^n\}$.}

Thanks to Theorem \ref{thm: blow-up 1}, we can adapt the proof of \cite[Theorem 3.2]{EspPet} with minor changes (some details are actually simpler in the present setting, since here we deal with a constant potential, differently to \cite{EspPet}). In adapting Theorem 3.2 from \cite{EspPet}, it is important to point out that any limit of $u_n$, given by Theorem \ref{thm: blow-up 1}, tends to $0$ at infinity. This fact is crucial in the proof of \eqref{39ep}.

Moreover, if the reference graph is unbounded, it is important to observe that $u_n(x) \to 0$ as $|x| \to +\infty$ on each half-line, since $u_n \in H^1(\cG)$ by assumption. This implies that, if $\{P_n^1\}$, \dots, $\{P_n^h\}$ are local maximum points of $u_n$, then there exists a maximum point on $\cG \setminus \bigcup_{i=1}^h B_{R \lambda_n^{-1/2}}(P_n^i)$.

\medskip

\emph{Step 2) Conclusion of the proof.} By \eqref{39ep}, for every $\eps \in (0,1)$ small, to be chosen later, there exist $R>0$ and $n_R \in \mathbb{N}$ large such that
\begin{equation}\label{max dn grande}
\max_{d_n(x) > R \lambda_n^{-1/2}} u_n(x) \le \lambda_n^{\frac1{p-2}} \eps, \quad \forall n \ge n_R.
\end{equation}
Thus, in the set $A_n:= \{d_n(x) > R \lambda_n^{-1/2}\}$, in addition to \eqref{max dn grande} we also have that
\begin{equation}\label{eq u lontano}
u_n'' = (\lambda_n- u_n^{p-2}) u_n \quad \implies \quad  -u_n'' + \frac{\lambda_n}{2} u_n \le 0
\end{equation}
provided that $\eps>0$ is small enough.

We want to exploit \eqref{max dn grande} and \eqref{eq u lontano} in a comparison argument, as in \cite{EspPet} (or \cite[Theorem 3.1]{Espetal}). However, the presence of the vertexes makes the argument a little bit more involved in our setting.

Let us denote by $\{{\rm v}_j\}_{j=1}^{h_1}$ the set of vertexes which are not included in one of the balls $B_{R \lambda_n^{-1/2}}(P^i_n)$ for large $n$. On any such vertex, by \eqref{max dn grande},
\begin{equation}\label{L1}
u_n({\rm v}_j) \leq \lambda_n^{\frac{1}{p-2}} \varepsilon.
\end{equation}
For any edge $\edge$, we consider the restriction of $u_n$ on $\edge \cap A_n$. Since $k$ is independent of $n$, $\edge \cap A_n$ consists in finitely many relatively open intervals (which may be unbounded, if $\cG$ is non-compact).

Let $I_n$ be any such \emph{bounded} interval; then the following alternative holds: $\pa I_n \cap \{{\rm v}_j\}_{j=1}^{h_1}$ can either be empty (case 1), or be a single vertex , say ${\rm v}_1$ (case 2), or be a pair of vertexes, say ${\rm v}_1$ and ${\rm v}_2$ (case 3).

Assume at first that case 1 holds. Then there exist two indexes $i, j \in\{1,\dots,k\}$ such that $\pa I_n$ consists in one point at distance $R \lambda_n^{-1/2}$ from $P^i_n$, and one point at distance $R \lambda_n^{-1/2}$ from $P^j_n$. Consider the function
\[
\phi_n(x) = e^{- \gamma \, \lambda_n^{\frac{1}{2}} |x-P^i_n|} + e^{- \gamma \, \lambda_n^{\frac{1}{2}} |x-P^j_n|},
\]
which solves $\phi_n''=\gamma^2 \lambda_n \phi_n$ in $I_n$. By taking $\gamma<1/4$, we have that
\[
-\phi_n''+\frac{\lambda_n}{2} \phi_n \ge 0 \quad \text{in $I_n$}.
\]
Moreover,
\[
\left.\left(e^{\gamma R} \lambda_n^{\frac{1}{p-2}} \phi_n - u_n\right)\right|_{\pa I_n} \ge \lambda_n^{\frac{1}{p-2}} (1-\eps) >0,
\]
and hence, by the comparison principle, we have that
\[
u(x) \le e^{\gamma R} \lambda_n^{\frac{1}{p-2}} \phi_n(x), \qquad \forall x \in I_n,
\]
which clearly implies the validity of the thesis on $I_n$ in this case.

If case 2 holds, then there exists an index $i \in \{1,\dots,k\}$ such that $\pa I_n$ consists in a point at distance $R \lambda_n^{-1/2}$ from $P^i_n$, plus the vertex ${\rm v}_1$. Arguing as before, it is not difficult to check that
\[
u(x) \le e^{\gamma R} \lambda_n^{\frac{1}{p-2}} e^{- \gamma \, \lambda_n^{\frac{1}{2}} |x-P^i_n|} + \lambda_n^{\frac{1}{p-2}} e^{- \gamma \, \lambda_n^{\frac{1}{2}} |x-{\rm v}_1|}, \quad \forall x \in I_n,
\]
which gives the thesis in case 2.

In case 3, an analogue argument ensures that
\[
u(x) \le  \lambda_n^{\frac{1}{p-2}} e^{- \gamma \, \lambda_n^{\frac{1}{2}} |x-{\rm} v_1|} + \lambda_n^{\frac{1}{p-2}} e^{- \gamma \, \lambda_n^{\frac{1}{2}} |x-{\rm v}_2|}, \quad \forall x \in I_n,
\]
whence the thesis follows once again.


Finally, let us consider the case when $I_n$ is an \emph{unbounded} interval of $\edge \cap A_n$. Then we only have two possibilities: either $\pa I_n$ consists in a point at distance $R \lambda_n^{-1/2}$ from $P^i_n$, or $\pa I_n$ consists in a vertex, say ${\rm v}_1$.

In the former case, we argue as before with the comparison function
\[
\psi_n(x) = e^{-\gamma R} \lambda_n^{\frac1{p-2}} e^{- \gamma \, \lambda_n^{\frac{1}{2}}|x-P^i_n|},
\]
where $\gamma<1/4$. In the latter one, we can use
\[
\psi_n(x) =  \lambda_n^{\frac1{p-2}} e^{- \gamma \, \lambda_n^{\frac{1}{2}}|x-{\rm v}_1|}.
\]

To sum up, slightly modifying the choice of the comparison functions, according to the structure of $\pa I_n$, it is possible to prove the validity of \eqref{estim far away} in all the possible cases.
\end{proof}

\section{Mountain pass solution for the original problem}\label{sec: ex mp}

In this section we complete the proof of the main existence result, Theorem \ref{thm: main ex}. Let $\mu \in (0, \mu_1)$. As already anticipated in Section \ref{sec: blow-up}, Proposition \ref{prop: ex for ae rho} gives a sequence of mountain pass critical points $u_{\rho_n} \in H^1_\mu(\cG)$ of $E_{\rho_n}(\cdot\,, \cG)$ on $H^1_\mu(\cG)$ with $\rho_n \to 1^-$ and $m(u_{\rho_n}) \le 2$. Moreover, the energy level $c_{\rho_n}$ is bounded, since
\[
E_1(\kappa_\mu, \cG) \le E_\rho(\kappa_\mu, \cG) \le c_{\rho} \le c_{1/2}, \qquad \forall \rho \in \left[\frac12, 1 \right]
\]
(the first and the second inequalities are proved in Lemma \ref{Lem2-20-1}; the third one follows directly from the monotonicity of $c_{\rho}$). Thus, Theorem \ref{thm: main ex} is a direct corollary of the next statement.

\begin{proposition}\label{thm: lambda bdd}
Let $\cG$ be a metric graph, $\{u_n\} \subset H^1(\cG)$ a sequence of solutions to \eqref{eq: blow-up} for some $\lambda_n \in \R$ and $\rho_n \to 1$. Suppose that
\[
\int_{\cG} |u_n|^2\,dx = \mu, \quad m(u_n) \le \bar k, \qquad \forall n,
\]
for some $\mu>0$ and $\bar k \in \mathbb{N}$, and that
\[
\text{the sequence of the energy levels $\{c_n:=E_{\rho_n}(u_n, \cG)\}$ is bounded}.
\]
Then the sequences $\{\lambda_n\} \subset \R$ and $\{u_n\} \subset H^1(\cG) $  must be bounded. In addition, $\{u_n\}$ is a (bounded) Palais-Smale sequence for $E_1(\cdot\,, \cG)$ constrained on $H^1_\mu(\cG)$.
\end{proposition}

\begin{proof}[Proof of Theorem \ref{thm: main ex}]
It is sufficient to apply Proposition \ref{thm: lambda bdd} on the sequence $\{u_{\rho_n}\}$ which, as observed, fulfills the assumptions.  Indeed, applying Lemma \ref{Embedding} we  then deduce that $u_n \to \bar u$ strongly in $H^1(\cG)$.
\end{proof}

\begin{proof}[Proof of Proposition \ref{thm: lambda bdd}]
Since
\[
\int_{\cG} \left(|u_{n}'|^2 + \lambda_n u_{n}^2\right)\,dx = \rho_n \int_{\cG} |u_{n}|^p\, dx,
\]
it follows that
\[
c_n =E_{\rho_n}(u_{n}, \cG) = \left( \frac12-\frac1p\right) \int_{\cG} |u_{n}'|^2\,dx - \frac{\lambda_n \mu}p;
\]
therefore
\begin{equation}\label{28011}
\left( \frac12-\frac1p\right) \int_{\cG} |u_{n}'|^2\,dx = c_n + \frac{\lambda_n \mu}p.
\end{equation}
This estimate gives the boundedness of $\{u_n\}$ in $H^1(\cG)$, provided that $\{\lambda_n\}$ is bounded (recall that $\{c_n\}$ is bounded as well).  Once the boundedness of $\{u_n\}$ in $H^1(\cG)$ is proved, and since $\rho_n \to 1$, the fact that it is a Palais-Smale sequence for $E_1(\cdot\,, \cG)$ constrained on $H^1_\mu(\cG)$ is straightforward.


Therefore, we only have to show that $\{\lambda_n\}$ is bounded. By contradiction, we suppose that this is not the case. By \eqref{28011}, we have that  $\lambda_n \to +\infty$, up to a subsequence. Thus, Theorems \ref{thm: blow-up 1} and \ref{thm: blow-up 2} hold for $u_n:=u_{\rho_n}$. For $\{P_n^1\}$, \dots, $\{P_n^k\}$ given by Theorem \ref{thm: blow-up 2}, Theorem \ref{thm: blow-up 1} ensures the existence of blow-up limits, which can be either defined on $\R$, or on a star graph $\cG_m$. In the rest of the proof:
\begin{itemize}
\item $\{v_{n}^i\}$ denotes the scaled sequence around $P_n^i$;
\item $V^i$ denotes the limit of $\{v^i_n\}$;
\item $\bar x^i_n$ denotes the global maximum point of $v^i_n$;
\item $\bar x^i$ denotes the global maximum point of $V^i$.
\end{itemize}
Then, for $R>0$, on one hand we have that
\begin{equation}\label{contrad 1}
\left| \lambda_n^{\frac{1}{2} - \frac{2}{p-2}} \int_{\cG} u_n^2\,dx - \sum_{i=1}^k \int_{B_R(\bar x_n^i)} (v_n^i)^2\,dx \right| \to +\infty
\end{equation}
(in the second integral, the ball $B_R(\bar x^i_n)$ is the ball in the scaled graph). Indeed, the first term inside the absolute value satisfies
\[
\lambda_n^{\frac{1}{2} - \frac{2}{p-2}} \int_{\cG} u_n^2\,dx = \lambda_n^{\frac{1}{2} - \frac{2}{p-2}} \mu \to +\infty,
\]
since $p>6$. While the second term is bounded, since by Theorem \ref{thm: blow-up 1}
\[
\sum_{i=1}^k \int_{B_R(\bar x_n^i)} (v_n^i)^2\,dx \to \int_{B_R(\bar x^i)} (V^i)^2\,dx,
\]
and it is the sum of a finite number of bounded integrals, being $V^i \in H^1(\cG_m)$.

On the other hand, by Theorem \ref{thm: blow-up 2}, for some positive constant $C$ which changes from one line to another,
\[
\begin{split}
\bigg| \lambda_n^{\frac{1}{2} - \frac{2}{p-2}} \int_{\cG} u_n^2\,dx &- \sum_{i=1}^k \int_{B_R(x_n^i)} (v_n^i)^2\,dx \bigg|  = \lambda_n^{\frac{1}{2} - \frac{2}{p-2}}\bigg| \int_{\cG} u_n^2\,dx- \sum_{i=1}^k\int_{B_{R \lambda_n^{-1/2}}(P_n^i)} u_n^2\,dx \bigg| \\
& = \lambda_n^{\frac{1}{2} - \frac{2}{p-2}} \int_{ \cG \setminus \bigcup_i B_{R \lambda_n^{-1/2}}(P_n^i)} u_n^2\,dx \\
& \le C_1 \lambda_n^{\frac{1}{2}} \sum_{i=1}^k\int_{ \cG \setminus \bigcup_i B_{R \lambda_n^{-1/2}}(P_n^i)}  e^{-C_2 \, \lambda_n^{\frac{1}{2}} \dist(x, P_n^i)}\,dx + C_1
\lambda_n^{\frac{1}{2}} \sum_{j=1}^{h}\int_{ \cG} e^{-C_2 \, \lambda_n^{\frac{1}{2}} \dist(x, {\rm v}_j )}\,dx
\\
& \le C \lambda_n^{\frac{1}{2}} \sum_{i=1}^k\int_{ \cG \setminus B_{R \lambda_n^{-1/2}}(P_n^i)}  e^{-C \, \lambda_n^{\frac{1}{2}} \dist(x, P_n^i)}\,dx + C
\lambda_n^{\frac{1}{2}} \sum_{j=1}^{h}\int_{ \cG} e^{-C \, \lambda_n^{\frac{1}{2}} \dist(x, {\rm v}_j )}\,dx
\\
& \le C \lambda_n^{\frac{1}{2}}  \int_{R \lambda_n^{-1/2}}^{+\infty} e^{- C  \, \lambda_n^{\frac{1}{2}} y}\,dy +   C \lambda_n^{\frac{1}{2}}  \int_{0}^{+\infty} e^{- C \, \lambda_n^{\frac{1}{2}} y}\,dy \\
&   \le C \int_R^{+\infty} e^{-C  z}\,dz + C \int_0^{\infty} e^{-C z} \, dz \\
& \le C e^{-C R} + C,
\end{split}
\]
By taking the limit as $n \to \infty$, we deduce that
\[
\limsup_{n \to \infty} \left| \lambda_n^{\frac{1}{2} - \frac{2}{p-2}} \int_{\cG} u_n^2\,dx - \sum_{i=1}^k \int_{B_r(\bar x_i)} (V^i)^2\,dx \right| \le C e^{-C R} + C,
\]
in contradiction with \eqref{contrad 1}.
\end{proof}

\end{document}